\documentclass[11pt]{amsart}
\usepackage{graphicx}
\usepackage[active]{srcltx}
 \makeatletter
\renewcommand*\subjclass[2][2000]{%
  \def\@subjclass{#2}%
  \@ifundefined{subjclassname@#1}{%
    \ClassWarning{\@classname}{Unknown edition (#1) of Mathematics
      Subject Classification; using '1991'.}%
  }{%
    \@xp\let\@xp\subjclassname\csname subjclassname@#1\endcsname
  }%
}
 \makeatother
\usepackage{enumerate,url,amssymb,  mathrsfs}

\newtheorem{theorem}{Theorem}[section]
\newtheorem{lemma}[theorem]{Lemma}
\newtheorem*{lemma*}{Lemma}
\newtheorem{proposition}[theorem]{Proposition}

\theoremstyle{definition}

\theoremstyle{remark}
\newtheorem{remark}[theorem]{Remark}

\numberwithin{equation}{section}

\def\XXint#1#2#3{{\setbox0=\hbox{$#1{#2#3}{\int}$}
\vcenter{\hbox{$#2#3$}}\kern-.5\wd0}}

\def\le{\leqslant}
\def\ge{\geqslant}
\begin{document}

\title[{{Solution operator of  Dirichlet problem in the unit ball}}]{{Solution operator of inhomogenuous  Dirichlet problem in the unit ball}}



 \author{ David Kalaj and Djordjije Vujadinovi\'c }
\address{ University of Montenegro, Faculty of Mathematics, Dzordza Va\v singtona  bb, 81000 Podgorica, Montenegro}
 \email{  djordjijevuj@t-com.me}
 \date{}
\begin{abstract}
In this paper we estimate norms of integral operator induced with Green function  related to the Poisson equation in the unit ball with vanishing boundary data.
 \end{abstract}

\maketitle

\section{Introduction and Notation}

We denote by $B^{n}$ and $S^{n-1}$ the unit ball and unit sphere in $R^n$ respectively. Throughout the paper we will assume that $n>2$ (the case $n=2$ has been already treated in \cite{Kalaj, Kalaj1}). By the vector norm $|\cdot|$ we consider $|x|=(\sum_{i=1}^{n}x_{i}^{2})^{\frac{1}{2}},$ and by the norm of an operator $T: X\rightarrow Y$ which acts between two normed  spaces $X$ and $Y$
 we mean
$$\|T\|=\sup\{\|Tx\|:\|x\|=1\}.$$
Let $P$ be the Poisson kernel, i.e. function
$$P(x,\eta)=\frac{1-|x|^2}{|x-\eta|^{n}},$$
and let $G$ be the Green function of the unit ball w.r.t. Laplace operator, i.e., the function
$$G(x,y)=c_{n}\left(\frac{1}{|x-y|^{n-2}}-\frac{1}{[x,y]^{n-2}}\right),$$ where \begin{equation}\label{cen}c_{n}=\frac{1}{(n-2)\omega_{n-1}},\end{equation} where $\omega_{n-1}$ is the Hausdorff measure of $S^{n-1}$ and $$ [x,y]:=|x|y|-y/|y||=|y|x|-x/|x||.$$ As it is known, both functions $P$ and $G$ are harmonic for $|x|<1$ with $x\neq y.$

Let $f:S^{n-1}\rightarrow R^{n}$ be a $L^1$  integrable function on the unit sphere $S^{n-1},$ and let $g:B^{n}\rightarrow R^{n}$ be $L^1$ integrable function in the unit ball. The solution of the  Poisson equation $\triangle u=g$ (in the sense of distributions),  in the unit ball, satisfying the boundary condition $u|_{S^{n-1}}=f\in L^{1}(S^{n-1})$ is given by
\begin{equation}
u(x)=P[f](x)-{\mathcal G}[g](x):=\int_{S^{n-1}}P(x,\eta)f(\eta)d\sigma(\eta)-\int_{B^{n}}G(x,y)g(y)dy,
\end{equation}
 $|x|<1.$ Here $d\sigma$ is the normalized Lebesgue $n-1$ dimensional measure of the Euclid sphere.

 We consider the Poisson equation with inhomogenous Dirichlet boundary condition \begin{equation}\label{eqpo}\left\{\begin{array}{rr}
                     \triangle u(x) &= g, x\in B^n\\
                        &u|_{\partial B^n}=0
                     \end{array}\right.\end{equation} where $g\in L^{p}(B^{n}),$ $p\ge 1$. The weak solution is then given by
 \begin{equation}
 u(x)=-{\mathcal G}[g](x)=-\int_{B^{n}}G(x,y)g(y)dy,|x|<1.
 \end{equation}
 The main goal of our paper is related to estimating various norms of the integral operator ${\mathcal G}$. We call it the \emph{solution operator} of Dirichlet's problem.  The compressive study of this problem for $n=2$ has been done by the first author in \cite{Kalaj}. In \cite{Kalaj1} it is considered its counterpart for \emph{differential operator} of Dirichlet's problem.  For some related results concerning the planar case we refer to the papers \cite{ah, duke, dost, dost1, dost2}. In \cite{akl}, Anderson, Khavinson and Lomonosov considered the $L^2$ norm of the operator $$\mathcal{N}[f](x)=:\frac{1}{(n-2)\omega_{n-1}}\int_{B^n}\frac{1}{|x-y|^{n-2}} f(y)dy.$$

 The following two results extend and generalize the corresponding results obtained in \cite{Kalaj} and \cite{akl}.
  \begin{theorem}\label{m1}
Let ${\mathcal G}:L^{p}(B)\rightarrow L^{\infty}(B)$, where  $p>n/2$. Then
$$\|\mathcal G\|=c_{n}\left(\frac{\pi ^{n/2} \Gamma(1+q) \Gamma\left(\frac{n-q(-2+n)}{-2+n}\right)}{\Gamma\left(1+\frac{n}{2}\right) \Gamma\left(\frac{n}{-2+n}\right)}\right)^{\frac{1}{q}},\enspace1\le q<\frac{n}{n-2} $$
where $\enspace n\geq 3$ and  $1/p+1/q=1$. In particular for $p=\infty$
$$\|\mathcal G\|_{\infty}=\frac{1}{2n}\enspace ( n\geq 3).$$
\end{theorem}
\begin{remark}
The particular case $p=\infty$ ($q=1$) of Theorem~\ref{m1}, is simply and follows from the following observation. Since the function $u(x)=-\frac{1}{2n}(1-|x|^2)$ represents unique solution of Poisson equation
$$\left\{\begin{array}{rr}
                     \triangle u(x) &= 1, x\in \Omega\\
                        &u|_{\partial \Omega}=0
                     \end{array}\right.,$$ it follows that for any integer $n,n\geq3$ we have
\begin{equation}
\|\mathcal G\|_{\infty}=\sup_{x\in B^n}\left|\int_{B^n}G(x,y)dy\right|=\frac{1}{2n}\sup_{x\in B^n}(1-|x|^2)=\frac{1}{2n}.
\end{equation}

\end{remark}
\begin{theorem}\label{hrt}
For $p\ge 1$, the operator $\mathcal{G}$ is a bounded operator of the space $L^p$ onto itself with the norm $\|\mathcal G\|_p$ satisfying the inequalities
$$\|{\mathcal G}\|_p\leq (2n)^{\frac{p-2}{p}}\lambda_1^{\frac{2(1-p)}{p}},\ \ 1\le p\le 2$$
and
 $$\|{\mathcal G}\|_p\leq \lambda_1^{-\frac{2}{p}}(2n)^{\frac{p-2}{p}},\quad 2\le p\le\infty$$ which reduces to an equality for $p=1,2,\infty$, where $\lambda_1=\lambda_1(B^n)$ is the the first eigenvalues of Dirichlet Laplacian of the unit ball defined in Subsection~\ref{eine}.
\end{theorem}
The proof of Theorem~\ref{m1} is postponed in section~4 and is  obtained via M\"obius transformations of the unit ball. It depends in Lemma~\ref{lema2}, which is somehow very involved and presents itself a subtle integral inequality. The proof of Theorem~\ref{hrt}, uses the eigenvalues of Dirichlet Laplacian and follows from Ries-Thorin interpolation theorem.

\section{Preliminaries}
   \subsection{Gauss hypergeometric function}
    Through the paper we will often use the properties of the hypergeometric functions. First of all, the hypergeometric function $F(a,b,c,t)={}_2F_{1} (a,b;c;t)$ is defined by the series expansion
  $$\sum_{n=0}^{\infty}\frac{(a)_{n}(b)_{n}}{n!(c)_n}t^{n},\enspace \mbox{for}\enspace|t|<1,$$
   and by the continuation elsewhere. Here $(a)_n$ denotes shifted factorial, i.e. $(a)_{n}=a(a+1)...(a+n-1)$ and $a$ is any real number.\\
   The following identities will be used in the proof of the main results of this paper:

    Euler's identity:
  \begin{equation}\label{en1}F(a,b;c;t)=(1-t^2)^{c-a-b}F(c-a,c-b;c;t),\enspace \mbox{Re}{(c)}>\mbox{Re}{(b)}>0,\end{equation}
  Pfaff's identity:
   \begin{equation}\label{en2}F(a,b;c;t)=(1-t^2)^{-a}F(a,c-b;c;\frac{t}{t-1}),\enspace \mbox{Re}{(c)}>\mbox{Re}{(b)}>0,\end{equation}
   Differentiation identity:
 \begin{equation}\label{en3}\frac{\partial}{\partial t}F(a,b;c;t)=\frac{ab}{c}F(a+1,b+1;c+1;t),\end{equation}
 and Kummer's Quadratic Transformation
\begin{equation}\label{en4}F\left(a,b;2b;\frac{4t}{(1+t)^2}\right)=(1+t)^{2a}F(a,a+\frac{1}{2}-b;b+\frac{1}{2};t^2),\end{equation}
 where above identity is true for every $t$  for which both series converge.

By using the Chebychev's inequality one can easily obtain the following inequality for Gamma function (see \cite{drag}).
\begin{proposition} Let $m$, $p$ and $k$ be real numbers with $m,p > 0$ and $p > k >- m$: If
\begin{equation}\label{(3.9)} k (p -m-k)\ge 0\ (\le 0)\end{equation}
then we have
\begin{equation}\label{(3.10)} \Gamma(p)\Gamma(m) \ge (\le ) \Gamma(p - k)\Gamma (m + k). \end{equation}
 \end{proposition}

 \subsection{M\"obius transformations of the unit ball}
 The set of isometries of the hyperbolic unit ball $B^n$ is a  Kleinian
subgroup of all M\"obius transformations of the extended space
$\overline{\mathbf{R}}^n$ onto itself denoted by
$\mathbf{Conf}(\mathbf{B}^n)=\mathbf{Isom}(\mathbf{B}^n)$. We refer
to the Ahlfors' book \cite{Ahl} for detailed survey to this class of
important mappings.  In general a
M\"obius transform $T_x:B^n\to B^n$ has the form
 \begin{equation}\label{mebius}z=T_{x}y=\frac{(1-|x|^2)(y-x)-|y-x|^{2}x}{[x,y]^2}.\end{equation} Then we have \begin{equation}\label{mebnorm} |T_{x}y|=\left|\frac{x-y}{[x,y]}\right|.\end{equation}
If $dy$ denotes the volume measure in the ball, because $y=T_{-x} z$ is a conformal mapping, in view of \eqref{mebnorm} we have
 \begin{equation}\label{deri}dy=\left(\frac{1-|x|^2}{[z,-x]^2}\right)^{n}dz.\end{equation}
 \subsection{Eigenvalues of Dirichlet Laplacian}\label{eine}

  First of all, it is known that there exist an orthonormal basis of $L^{2}(B^{n})$ consisting of eigenfunctions  $(\varphi_{n})_{n}$ of Dirichlet Laplacian
\begin{equation}\label{popo}\left\{
           \begin{array}{ll}
             -\Delta u = \lambda  u, & z\in B^{n} \\
             u|_{\partial B^{n}} = 0&
           \end{array}
         \right.\end{equation}
  with corresponding eigenvalues $\lambda_{1}<\lambda_{2}\leq...\leq\lambda_{n}...$ The functions $\varphi_{n}$ are real valued.

  It is well known that $\lambda_1(B^n)$ is given by the square of the first positive zero of the Bessel function
$J_{(n-1)/2}(t)$ of the first kind of order $\alpha=(n-1)/2$: \begin{equation}\label{bessel}J_\alpha(t)=\sum_{m=0}^\infty\frac{(-1)^m}{m!\Gamma(m+\alpha+1)}\left(\frac{t}{2}\right)^{2m+\alpha}.\end{equation}

 \section{The main lemma}
\begin{lemma}\label{lema2}

Let $$I(t)=(1-t^2)^{n-q(n-2)}\int_{0}^{1}\frac{(1-r^{n-2})^{q}r^{n-q(n-2)-1}}{(1-r^{2}t^{2})^{n-q(n-2)+1}}dr,\enspace 0\leq t< 1,$$
  where $n\geq 3$ is a natural number and $1<q<\frac{n}{n-2}.$ Then the maximal value of function $I(t)$ is attained for $t=0,$ i.e.,
\begin{equation}
\begin{split}
\max_{0\leq t< 1} I(t)&=I(0)=\int_{0}^{1}(1-r^{n-2})^{q}r^{n-q(n-2)-1}dr\\&=\frac{\Gamma(1+q)\Gamma\left(\frac{n-q(n-2)}{n-2}\right)}{(n-2)\Gamma\left(1+q+\frac{n-q(n-2)}{n-2}\right)}
\end{split}
\end{equation}

\end{lemma}
\begin{proof}
At the beginning we will observe the case $n>3.$
For $a=n-q(n-2)$ we have $0<a<2$ and the next expansion
\begin{equation}
\begin{split}
I(t)&=(1-t^2)^{a}\int_{0}^{1}\frac{(1-r^{n-2})^{\frac{n-a}{n-2}}r^{a-1}}{(1-r^2t^2)^{a+1}}dr\\
&=(1-t^2)^{a}\sum_{k=0}^{\infty}\frac{\Gamma(k+a+1)}{\Gamma(a+1)k!}t^{2k}\int_{0}^{1}(1-r^{n-2})^{\frac{n-a}{n-2}}r^{2k+a-1}dr\\
&=\frac{\Gamma(2+\frac{2-a}{n-2})(1-t^2)^{a}}{(n-2)\Gamma(a+1)}\sum_{k=0}^{\infty}\frac{\Gamma(k+a+1)\Gamma(\frac{a+2k}{n-2})}{\Gamma\left(\frac{2(k+n-1)}{n-2}\right)k!}t^{2k}.
\end{split}
\end{equation}
Assume that $n\ge 3 $ and $k\ge 0$. Let
 $$K=\frac{2k}{n-2},\ \ \  M=2+\frac{2k+a}{n-2},\ \ \ P=2+\frac{2}{n-2}.$$ From  \eqref{(3.10)} we have
\begin{equation}\label{mk}\Gamma(M)\Gamma(P)\le \Gamma(M-K)\Gamma(P+K).\end{equation}  By using the formula $\Gamma(x+1)=x\Gamma (x)$ and \eqref{mk}, we have
\[\begin{split}\label{key}
\frac{\Gamma\left(\frac{a+2 k}{n-2}\right)}{\Gamma\left(\frac{2 (k+n-1)}{n-2}\right)}&=\frac{\Gamma(2+\frac{a+2k}{n-2})}{\Gamma(2+\frac{2+2k}{n-2})}\frac{1}{\left(\frac{a+2 k}{n-2}\right)\left(\frac{a+2 k}{n-2}+1\right)}\\
&\leq \frac{\Gamma(2+\frac{a}{n-2})}{\Gamma(2+\frac{2}{n-2})}\frac{1}{\left(\frac{a+2 k}{n-2}\right)\left(\frac{a+2 k}{n-2}+1\right)}.
\end{split}\]
  We obtain for $a\in(0,2)$
\[\begin{split}\frac{I(t)}{\frac{\Gamma\left(2+\frac{2-a}{n-2}\right)}{n-2}}:\frac{\Gamma(2+\frac{a}{n-2})}{\Gamma(2+\frac{2}{n-2})}&\le \frac{\left(1-t^2\right)^a}{\Gamma(a+1)}\sum_{k=0}^\infty
\frac{\Gamma(a+k+1)} {\Gamma(1+k)}\frac{t^{2k}}{\left(\frac{a+2 k}{n-2}\right)\left(\frac{a+2 k}{n-2}+1\right)}\\
&= \frac{(n-2) (1-t^2)^a}{a  } F\left(\frac{a}{2},1+a,\frac{2+a}{2},t^2\right)\\&-\frac{(n-2) (1-t^2)^a}{ (n+a-2) }F\left(1+a,\frac{1}{2} (n+a-2),\frac{a+n}{2},t^2\right)\\
&=  \frac{(n-2)}{{a  }}F\left(1,-\frac{a}{2},1+\frac{a}{2},t^2\right)\\& -\frac{(n-2)a}{{a (n+a-2) }}F\left(1,\frac{1}{2} (n-a-2),\frac{a+n}{2},t^2\right)\\&:=J(t).\end{split}\]
The last expression for the function $J(t)$ was obtained by using the identity \eqref{en4}.
Further we have
\begin{equation}
\begin{split}\label{io}
\frac{\partial J(t)}{\partial t}&=\frac{-2t(n-2)}{a+2} F\left(2,\frac{2-a}{2},2+\frac{a}{2},t^2\right)\\&-\frac{2t(n-2)(n-a-2)}{ (n+a-2)(a+n) }F\left(2,1+\frac{1}{2} (n-a-2),1+\frac{a+n}{2},t^2\right)\\
&<0.
\end{split}
\end{equation}
We conclude that the maximal value of the function $I(t)$  for $t=0$ is attained.

 In order to prove the special case $n=3$, $1<q<3,$ of Lemma~\ref{lema2}, we should notice that
\begin{equation}
\begin{split}
\max_{0\leq t< 1}I(t)&=\max_{0\leq t< 1}(1-t^2)^{3-q}\int_{0}^{1}\frac{(1-r)^{q}r^{2-q}}{(1-r^2t^2)^{4-q}}dr\\
&\leq \max_{0\leq t< 1}(1-t^2)^{3-q}\int_{0}^{1}\frac{(1-r)^{q}r^{2-q}}{(1-rt^2)^{4-q}}dr.
\end{split}
\end{equation}
Put $$J(t):=(1-t^2)^{3-q}\int_{0}^{1}\frac{(1-r)^{q}r^{2-q}}{(1-rt^2)^{4-q}}dr, 0\leq t< 1.$$
By using the Taylor expansion we obtain
\begin{equation}
J(t)=\frac{\Gamma(1+q)\Gamma(3-q)}{6}(1-t^{2})^{3-q}F\left(4-q,3-q,4;t^2\right), 0\leq t< 1.
\end{equation}
By using \eqref{en1} and \eqref{en2} respectively on expression for  $J(t)$ we have
\begin{equation}
J(t)=\frac{\Gamma(1+q)\Gamma(3-q)}{6}F\left(q,3-q,4;\frac{t^2}{t^{2}-1}\right), 0\leq t< 1.
\end{equation}
So,
\begin{equation}
\begin{split}
\max_{0\leq t< 1}J(t)&=\frac{\Gamma(1+q)\Gamma(3-q)}{6}\max_{0\leq t< 1}F\left(q,3-q,4;\frac{t^2}{t^{2}-1}\right)\\
&=\frac{\Gamma(1+q)\Gamma(3-q)}{6}\max_{0\leq t< 1}F\left(q,3-q,4;0\right).
\end{split}
\end{equation}
The last equality is a consequence of the fact that $\frac{t^2}{t^2-1}<0$ and that coefficients  $$\frac{(q)_{k}(3-q)_{k}}{(1)_k(4)_{k}}$$ of the hypergeometric function $$F\left(q,3-q,4;\frac{t^2}{t^{2}-1}\right)$$ are  decreasing with respect to $k\geq 1.$\\
So,
\begin{equation}
\max_{0\leq t< 1}I(t)=I(0)=\int_{0}^{1}(1-r)^{q}r^{2-q}dr=\frac{\pi q(1-q)(2-q)}{6\sin{\pi q}}.
\end{equation}
\end{proof}

\section{Proof of Theorem~\ref{m1}}
We start this section with an easy lemma.

\begin{lemma}\label{lopi}
 Let  $\|{\mathcal G}\|:=\|{\mathcal G}:L^{p}(B^{n})\rightarrow L^{\infty}(B^{n})\|$ for $ p>\frac{n}{2}.$ Then
$$\|\mathcal G\|=\sup_{x\in B^n}\left(\int_{B^{n}}\left|G(x,y)\right|^{q}dy\right)^{\frac{1}{q}},\enspace \frac{1}{p}+\frac{1}{q}=1.$$
\end{lemma}
\begin{proof}
Let $u(x)={\mathcal G}[g](x),\enspace g\in L^{p}(B).$ H\"{o}lder inequality implies
$$\|u\|_{\infty}\leq\sup_{x\in B}\left(\int_{B}\left|G(x,y)\right|^{q}dy\right)^{\frac{1}{q}}\left(\int_{B}|g(y)|^{p}dy\right)^{\frac{1}{p}},$$
i.e.,
$$\|\mathcal G\|\leq\sup_{x\in B^n}\left(\int_{B^n}\left|G(x,y)\right|^{q}dy\right)^{\frac{1}{q}}.$$
On the other hand, there exist $x_{0}\in B^n$ so that $$\left(\int_{B^n}\left|G(x_{0},y)\right|^{q}dy\right)^{\frac{1}{q}}>\sup_{x\in B^n}\left(\int_{B^n}\left|G(x,y)\right|^{q}dy\right)^{\frac{1}{q}}-\epsilon.$$
We fix  $x_{0}\in B^n.$ Let us consider the function $$g(y)=\frac {(G(x_{0},y))^{q-1}}{\|(G(x_{0},y))^{q-1}\|_{p}}.$$ Then
\begin{equation}
\begin{split}
\|\mathcal G\|&\geq |{\mathcal G}[g](x_{0})|\\&=\left(\int_{B^n}\left|G(x_{0},y)\right|^{q}dy\right)^{-\frac{1}{p}}\int_{B^n}|G(x_{0},y)|^{q}dy\\
&=\left(\int_{B^n}\left|G(x_{0},y)\right|^{q}dy\right)^{\frac{1}{q}}\\&>\sup_{x\in B^n}\left(\int_{B^n}\left|G(x,y)\right|^{q}dy\right)^{\frac{1}{q}}-\epsilon,
\end{split}
\end{equation}
 i.e., $$\|\mathcal G\|=\sup_{x\in B^n}\left(\int_{B^n}\left|G(x,y)\right|^{q}dy\right)^{\frac{1}{q}}.$$
\end{proof}
\begin{proof}[Proof of Theorem~\ref{m1}]
We divide the proof into two cases.

(i) This case includes the following range for $(n,q)$: $n> 3,$ with  $1<q<\frac{n}{n-2}$ and $n=3$ with $q\in (2,3).$
According to Lemma~\ref{lopi}, $$\|\mathcal G\|=\sup_{x\in B^n}\left(\int_{B^n}\left|G(x,y)\right|^{q}dy\right)^{\frac{1}{q}},q>1.$$
Further  we have
\begin{equation}
\|\mathcal G\|^q=c_{n}^{q}\sup_{x\in B^{n}}\int_{B^{n}}\frac{1}{|x-y|^{q(n-2)}}\left|1-\left|\frac{x-y}{[x,y]}\right|^{n-2}\right|^{q}dy,
\end{equation} where $c_{n}$ is defined in \eqref{cen}.
We use the change of variable $z=T_{x}y$ i.e. $T_{-x}z=y,$ in the previous integral where $T_{x}y$ M\"{o}bius transform defined in \eqref{mebius}.
By \eqref{deri}, denoting $t=|x|$, we obtain,
\[
\begin{split}
&\sup_{x\in B^n}\int_{B^n}\left|G(x,y)\right|^{q}dy\\&=\sup_{x\in B^n}c_{n}^{q}\int_{B^n}\frac{1}{|x-T_{-x}z|^{q(n-2)}}|1-|z|^{n-2}|^{q}\frac{(1-t^2)^n}{[z,-x]^{2n}}dz\\
&=c_{n}^{q}\sup_{x\in B^n}(1-t^2)^{n}\int_{B^n}\frac{(1-|z|^{n-2})^q}{\left|\frac{x[z,-x]^{2}-(1-t^2)(x+z)-|x+z|^{2}x}{[z,-x]^2}\right|^{q(n-2)}}\frac{dz}{[z,-x]^{2n}}\\
&=c_{n}^{q}\sup_{x\in B^n}(1-t^2)^{n}\int_{B^n}\frac{(1-|z|^{n-2})^q}{|z|^{q(n-2)}\left|\frac{1-t^2}{[z,-x]}\right|^{q(n-2)}}\frac{dz}{[z,-x]^{2n}}
\end{split}
\]
\[
\begin{split}
&=c_{n}^{q}\sup_{x\in B^n}(1-t^2)^{n-q(n-2)}\int_{B^n}\left(\frac{1-|z|^{n-2}}{|z|^{n-2}}\right)^{q}[z,-x]^{q(n-2)-2n}dz\\
&=c_{n}^{q}\sup_{x\in B^n}(1-t^2)^{n-q(n-2)}\int_{0}^{1}\frac{(1-r^{n-2})^{q}}{r^{q(n-2)+1-n}}dr\int_{S}\frac{d\xi}{|rx+\xi|^{2n-q(n-2)}}\\
&=c_{n}^{q}\sup_{x\in B^n }(1-t^2)^{a}\int_{0}^{1}\frac{(1-r^{n-2})^{q}}{r^{1-a}}dr \int_{S}\frac{d\xi}{(r^{2}t^{2}+2rt\xi_{1}+1)^{\frac{n+a}{2}}}\\
&=c_{n}^{q}C_{n}\sup_{x\in B^n}(1-t^2)^{a}\int_{0}^{1}\frac{(1-r^{n-2})^{q}}{r^{1-a}}dr
\int_{-1}^{1}\frac{(1-s^2)^{\frac{n-3}{2}}}{(r^{2}t^{2}+2rts+1)^{\frac{n+a}{2}}}ds,
\end{split}
\]
where $$a=n-q(n-2),\quad C_{n}=\frac{\omega_{n-1}\Gamma(n-1)}{2^{n-2}\Gamma^{2}(\frac{n-2}{2})}$$ and in last two equalities it was assumed without loss of generality that $x=te_1,\xi=(\xi_1,...,\xi_n).$
If we take change of variable $$\tau =\frac{1-s}{2}$$ in the previous integral we have
\begin{equation}
\begin{split}
\|\mathcal G\|^q:c_{n}^{q}&=C_{n}\sup_{x\in B^n}(1-t^2)^{a}\int_{0}^{1}\frac{(1-r^{n-2})^{q}}{r^{1-a}}dr\int_{-1}^{1}\frac{(1-s^2)^{\frac{n-3}{2}}}{(r^{2}t^{2}+2rts+1)^{\frac{n+a}{2}}}ds\\
&=2^{n-2}C_n\sup_{x\in B^n}(1-t^2)^{a}\int_{0}^{1}\frac{(1-r^{n-2})^{q}r^{a-1}}{(1+rt)^{n+a}}dr
\int_{0}^{1}\frac{\tau^{\frac{n-3}{2}}(1-\tau)^{\frac{n-3}{2}}}{(1-\frac{4rt\tau}{(1+rt)^2})^{\frac{n+a}{2}}}d\tau.
\end{split}
\end{equation}
On the other hand, for fixed $r$ we have $\frac{4rt}{(1+rt)^2}<1$ and
\begin{equation}
\begin{split}
&\int_{0}^{1}\frac{\tau^{\frac{n-3}{2}}(1-\tau)^{\frac{n-3}{2}}}{(1-\frac{4rt\tau}{(1+rt)^2})^{\frac{a+n}{2}}}d\tau\\&=\sum_{k=0}^{\infty}\frac{\Gamma(\lambda+k)}{k!\Gamma(\lambda)}\left(\frac{4rt}{(1+rt)^2}\right)^{k}\int_{0}^{1}\tau^{k+\frac{n-3}{2}}(1-\tau)^{\frac{n-3}{2}}d\tau\\
&=\Gamma\left(\frac{n-1}{2}\right)\sum_{k=0}^{\infty}\frac{\Gamma(\lambda+k)\Gamma(k+\frac{n-3}{2}+1)}{k!\Gamma(\lambda)\Gamma(n-1+k)}\left(\frac{4rt}{(1+rt)^2}\right)^{k}\\
&=\frac{\Gamma^2(\frac{n-1}{2})}{\Gamma(n-1)}F\left(\lambda, \frac{n-1}{2};n-1; \frac{4rt}{(1+rt)^2}\right),
\end{split}
\end{equation}
where $\lambda=\frac{n+a}{2}.$\\
 By using  Kummer quadratic transformation and Euler's transformation for hypergeometric functions, for $t=|x|$, we obtain
\begin{equation}
\begin{split}\label{am}
&\sup_{x\in B^n}(1-t^2)^{a}\int_{0}^{1}\frac{(1-r^{n-2})^{q}r^{a-1}}{(1+rt)^{n+a}}F\left(\lambda, \frac{n-1}{2};n-1; \frac{4rt}{(1+rt)^2}\right)dr\\
&=\sup_{x\in B^n}(1-t^2)^{a}\int_{0}^{1}(1-r^{n-2})^{q}r^{a-1}F\left(\frac{n+a}{2},\frac{a+2}{2};\frac{n}{2}; r^2t^2\right)dr\\
&= \sup_{x\in B^n}(1-t^2)^{a}
\int_{0}^{1}(1-r^{n-2})^{q}r^{a-1}(1-r^2t^2)^{-a-1}\mathcal{F}(rt)dr\\
&\leq \sup_{x\in B^n}(1-t^2)^{a}
\int_{0}^{1}(1-r^{n-2})^{q}r^{a-1}(1-r^2t^2)^{-a-1}\max_{t\leq 1}\mathcal{F}(rt)dr,
\end{split}
\end{equation}
where $$\mathcal{F}(s)=F\left(-\frac{a}{2},\frac{q(n-2)-2}{2};\frac{n}{2}; s^2\right).$$
 Then by using the identity for the derivative of hypergeometric function
we obtain
\begin{equation}
\begin{split}
&\frac{\partial}{\partial t}F\left(-\frac{a}{2},\frac{q(n-2)-2}{2};\frac{n}{2}; r^2t^2\right)\\
&=-2r^2t\frac{2}{n}\frac{a}{2}\frac{q(n-2)-2}{2}F\left(\frac{q(n-2)-n+2}{2},\frac{q(n-2)}{2};\frac{n+2}{2}; r^2t^2\right)<0,
\end{split}
\end{equation}
 for any $t\in [0,1],$ which implies
\begin{equation}
\begin{split}
\max_{|x|\leq 1}F\left(-\frac{a}{2},\frac{q(n-2)-2}{2};\frac{n}{2}; r^2|x|^2\right)
=F\left(-\frac{a}{2},\frac{q(n-2)-2}{2};\frac{n}{2}; 0\right).
\end{split}
\end{equation}

Finally, according to Lemma~\ref{lema2},  for  $n>3$ the maximal value of the function
\[\begin{split}\mathcal{I}(x)&=\int_{B^{n}}\left|G(x,y)\right|^{q}dy\\&=c_{n}^{q}(1-|x|^2)^{a}
\int_{0}^{1}(1-r^{n-2})^{q}r^{a-1}dr\int_{S}\frac{d\xi}{|rx+\xi|^{2n-q(n-2)}}\end{split}\] is attained for $x=0.$ So,
\begin{equation}
\begin{split}
\|G\|^q:c_n^q=&\sup_{x\in B^n}(1-|x|^2)^{a}\int_{0}^{1}(1-r^{n-2})^{q}r^{a-1}dr\int_{S}\frac{d\xi}{|rx+\xi|^{n+a}}\\
&=\omega_{n-1}\sup(1-|x|^2)^{a}\int_{0}^{1}\frac{(1-r^{n-2})^{q}r^{a-1}}{(1-r^2|x|^2)^{a+1}}\mathcal{F}(r|x|)dr\\
&=\omega_{n-1}\int_{0}^{1}(1-r^{n-2})^{q}r^{n-q(n-2)-1}drF\left(\frac{n+a}{2},\frac{q(n-2)-2}{2};\frac{n}{2}; 0\right)\\
&=\omega_{n-1}\int_{0}^{1}(1-r^{n-2})^{q}r^{a-1}dr=\frac{\omega_{n-1}\Gamma(1+q)\Gamma(\frac{n-q(n-2)}{n-2})}{(n-2)\Gamma(1+q+\frac{n-q(n-2)}{n-2})}.
\end{split}
\end{equation}

(ii) The  case $n=3$ with $1<q\leq 2.$ It is clear that
 $$\mathcal{I}(x)=\int_{B^{3}}\left|G(x,y)\right|^{q}dy=\frac{1}{(2\pi)^q}\int_{B^3}\left(\frac{1}{|x-y|}-\frac{1}{[x,y]}\right)^q dy,$$
 and that the same transforms for $I(x)$ as in the previous general case give
 $$\mathcal{I}(x)=c_{3}(1-x^2)^{3-q}\int_0^1(1-r)^q r^{2-q}F[(6-q)/2,(5-q)/2,3/2,r^2x^2]dr,$$ where $c_{3}$ is appropriate constant as in general case.  Put $t=|x|$.  We can represent $\mathcal{I}(x)$  as
$$\mathcal{I}(x)=c_{3}\int_0^1 \frac{(1-r)^q r^{1-q} \left(1-t^2\right)^{3-q} \left((1-r t)^{-4+q}-(1+r t)^{-4+q}\right)}{2 (4-q) t}dr.$$
So, $$\mathcal{I}(x)=c_{3}\frac{\left(1-t^2\right)^{3-q}}{2 (4-q) t}\sum_{n=0}^\infty t^n\int_0^1 (1 - r)^q r^{1 - q} (r^n-(-r)^n)  \binom{-4 + q}{ n} dr, $$
and this implies  $$\mathcal{I}(x)=c_{3}\frac{\left(1-t^2\right)^{3-q}}{2 (4-q) t}\sum_{n=0}^\infty\frac{\left(-1+e^{i n \pi }\right) \binom{-4+q}{n} \Gamma(2+n-q) \Gamma(1+q)}{\Gamma(3+n)}t^n.$$
Thus $$\mathcal{I}(x)=c_{3}\frac{\pi  (-1+q) q \left(1-t^2\right)^{3-q} (F(2-q,4-q;3;t)-{}F(2-q,4-q;3;-t))}{4\sin(\pi q) (4-q) t}.$$
Let $$c'(q):=c_{3}\frac{2^{-q} \pi ^{2-q}(-1 + q) q } {(4-q)\sin(\pi q)}. $$
Then $$\mathcal{I}(x)/|c'|\le I_1(x)=\frac{(1-t^2)}{t}\left(F(2-q,4-q;3;t)-F(2-q,4-q;3;-t)\right)$$ for $1<q<2$ and
$$I_1(x)=a_0+\sum_{n=1}^\infty {a_n}t^n$$ where $a_0>0$ and $$a_n=\frac{2 \left(1+(-1)^n\right) \Gamma(3+n-q) (-(n-q)! \Gamma(4+n)+\Gamma(n) \Gamma(5+n-q))}{\Gamma(n) \Gamma(2+n) \Gamma(4+n) \Gamma(2-q) \Gamma(4-q)}.$$ Further $a_n\le 0 $ because $$\frac{(1+n-q) (2+n-q) (3+n-q) (4+n-q)}{n (1+n) (2+n) (3+n)}\le 1,$$ which again implies that maximal value of the function $\mathcal{I}(x)$ is attained for the $x=0.$
This finishes the proof of Theorem~\ref{m1}.
\end{proof}
\section{Proof of Theorem~\ref{hrt}}
Let $\Omega$ be a domain of $\mathbf{R}^n$ and let $|\Omega|$ be its volume. For $\mu \in (0,1]$ define the operator $V_{\mu}$ on the space $L^{1}(\Omega)$ by Riesz potential
$$(V_{\mu}f)(x)=\int_{\Omega}|x-y|^{n(\mu-1)}f(y)dy.$$ The operator $V_{\mu}$  is defined for any $f\in L^{1}(\Omega)$ and  $V_{\mu}$ is bounded on $ L^{1}(\Omega),$ or more generally we have the next lemma.
\begin{lemma}\cite[p. 156-159]{gt}.\label{lema12}
Let  $V_{\mu}$ be defined on the $ L^{p}(\Omega),p>0.$ Then  $V_{\mu}$ is continuous as a mapping
$V_{\mu}: L^{p}(\Omega)\rightarrow  L^{q}(\Omega),$ where $1\leq q\leq \infty,$ and
$$0\leq \delta=\delta(p,q)=\frac{1}{p}-\frac{1}{q}<\mu.$$
Moreover, for any $f\in  L^{p}(\Omega)$
$$\|V_{\mu}f\|_{q}\leq \left(\frac{1-\delta}{\mu-\delta}\right)^{1-\delta}(\omega_{n-1}/n)^{1-\mu}|\Omega|^{\mu-\delta}\|f\|_{p}.$$
\end{lemma}
\begin{remark}
If instead of Riesz potential we consider the solution operator for a domain $\Omega$ with finite volume, then the operator is in general non-bounded. However it is bounded, if the boundary of $\Omega$ is enough regular. See \cite{jk} for an essential approach to the solution of this problem.
\end{remark}
\begin{theorem}
Let  $\|\mathcal G\|_1:=\|{\mathcal G}:L^{1}(B)\rightarrow L^{1}(B)\|,$ then

$$\|\mathcal G\|_1= \frac{1}{2n}.$$

\end{theorem}

\begin{proof}
According to Theorem~\ref{m1} we have $$\|\mathcal G\|_{L^{\infty}\rightarrow L^{\infty}}=\frac{1}{2n}.$$
On the other hand, Lemma~\ref{lema12} states that $\mathcal G:L^{1}\rightarrow L^{1}$ is bounded.  Then
$$\|\mathcal G\|_{L^{1}\rightarrow L^{1}}=\|\mathcal G^{\ast}\|_{L^{\infty}\rightarrow L^{\infty}},$$
where $\mathcal G^{\ast}$ is appropriate adjoint  operator. Since $$\mathcal G^{\ast}f(x)=\int_{B^n}\overline{G(y,x)}f(y)dy=\int_{B^n}G(x,y)f(y)dy,f\in L^{\infty}(B),$$ we have $$\|\mathcal G\|_{L^{1}\rightarrow L^{1}}=\|\mathcal G\|_{L^{\infty}\rightarrow L^{\infty}}.$$
\end{proof}
In the sequel we are going to observe Hilbert case $p=2,\mathcal G:L^{2}(B)\rightarrow L^{2}(B).$ It is well-known that $\mathcal{G}^{-1}=-\triangle$ on the Sobolev space $H_0^1(\Omega)$, so the Hilbert norm  $\mathcal{G}$ is precisely the reciprocal value of the norm of $-\triangle$ (c.f. \cite{akl}). So we have the following theorem, whose proof is included for the sake of completeness.
\begin{theorem}
Let $\|\mathcal{G}\|_2:=\|\mathcal G:L^{2}(B^n)\rightarrow L^{2}(B^n)\|,$ then
$$\|\mathcal G\|_2=\frac{1}{\lambda_{1}}.$$ Thus
\begin{equation}\label{ok}
\|\mathcal G g\|_{2}\leq\frac{1}{\lambda_{1}}\|g\|_{2},\enspace g\in L^{2}(B^n).
\end{equation}
Equality is attained in \eqref{ok} for $g(x)=c\varphi_{1}(x),a.e.\enspace x\in B^n$  where $c$ is a real constant.
\end{theorem}
\begin{proof}
If $f\in L^{2}(B^n),$ then under the previous notation $$f(x)=\sum_{k=1}^{\infty}\left<f,\varphi_{k}\right>\varphi_{k}(x).$$ Since $\mathcal G$ is bounded, we have
$$\mathcal G[f]=\sum_{k=1}^{\infty}\left<f,\varphi_{k}\right>\mathcal G[\varphi_{k}].$$
Also,
$$\mathcal G[\varphi_{k}]=\frac{1}{\lambda_{k}}\mathcal G[\triangle\varphi_{k}]=-\frac{1}{\lambda_{k}}\varphi_{k}.$$  The fact that  $(\varphi_{k})$ is orthonormal implies
$$\|\mathcal G f\|_{2}^{2}=\sum_{k=1}^{\infty}\frac{|\left<f,\varphi_{k}\right>|^{2}}{\lambda_{k}^{2}}.$$
Since $\lambda_1$ is a simple eigenvalue and $0<\lambda_{1}<\lambda_{2}\leq...,$ we have
$$ \|\mathcal G f\|_{2}\leq \frac{1}{\lambda_1}\|f\|_2.$$ Finally,
$$\|\mathcal G \|_{2}=\frac{1}{\lambda_1}.$$
\end{proof}
By using the Ries-Thorin interpolation theorem \cite{thorin}, we obtain the following estimates of the norm of  the operator $\mathcal G:L^p\to L^p.$

Let us denote by $\|{\mathcal G}\|_{L^{1}\rightarrow L^{1}}=\|{\mathcal G}\|_{L^{\infty}\rightarrow L^{\infty}}=\|\mathcal G\|_{1}$ and
$\|{\mathcal G}\|_{L^{2}\rightarrow L^{2}}=\|{\mathcal G}\|_{2}.$ Then
$$\|{\mathcal G}\|_p\leq \|{\mathcal G}\|_{1}^{\frac{2-p}{p}}\|{\mathcal G}\|_{2}^{\frac{2(p-1)}{p}}=(2n)^{\frac{p-2}{p}}\lambda_1^{\frac{2(1-p)}{p}},$$
where $\|{\mathcal G}\|_p$ represents the norm of the operator ${\mathcal G}:L^{p}(B^{n})\rightarrow L^{p}(B^n)$, $1<p<2.$
Similarly,
 $$\|{\mathcal G}\|_p\leq \|{\mathcal G}\|_{2}^{\frac{2}{p}}\|{\mathcal G}\|_{1}^{\frac{p-2}{p}}=\lambda_1^{-\frac{2}{p}}(2n)^{\frac{p-2}{p}},$$
 where ${\mathcal G}:L^{p}(B^{n})\rightarrow L^{p}(B^n),$ $2<p<\infty.$
This yields the proof of Theorem~\ref{hrt}.

\end{document}